\documentclass[12pt,reqno]{amsart}
\usepackage{tikz}
\usepackage{url}
\usepackage{tikzit}

\tikzstyle{empty}=[fill=none, draw=black, shape=circle, tikzit fill=white]
\tikzstyle{full}=[fill=black, draw=black, shape=circle, tikzit shape=circle]
\tikzstyle{red}=[fill=red, draw=red, shape=circle, tikzit shape=circle]
\tikzstyle{gray}=[fill=gray, draw=gray, shape=circle, tikzit shape=circle]
\tikzstyle{pink}=[fill=white, draw=black, shape=circle, tikzit draw=black, tikzit fill={rgb,255: red,160; green,6; blue,255}]
\tikzstyle{green}=[fill=white, draw=black, shape=circle, tikzit fill={rgb,255: red,0; green,181; blue,12}, tikzit draw=black]
\tikzstyle{orange}=[fill=white, draw=black, shape=circle, tikzit fill={rgb,255: red,255; green,174; blue,34}, tikzit draw=black]
\tikzstyle{blue}=[fill={rgb,255: red,64; green,0; blue,255}, draw=black, shape=circle]

\tikzstyle{edgeD}=[-, dashed=dashed, draw=black, fill=none]
\tikzstyle{edgeGreen}=[-, draw={rgb,255: red,0; green,181; blue,12}]
\tikzstyle{edgeOrange}=[-, draw={rgb,255: red,255; green,174; blue,34}]
\tikzstyle{edgeRed}=[-, draw=red]
\tikzstyle{edgeB}=[-, draw=black, fill=none]
\tikzstyle{edgeFill}=[fill={rgb,255: red,230;green,230;blue,230},draw=none]

\usepackage{float}

\def\NZQ{\mathbb}               
\def\NN{{\NZQ N}}

\def\ZZ{{\NZQ Z}}

\def\FF{{\NZQ F}}

\def\G{{\mathcal G}}
\def\F{{\mathcal F}}

\def\pd{\textup{proj}\phantom{.}\!\textup{dim}}

\def\opn#1#2{\def#1{\operatorname{#2}}} 
\opn\chara{char} \opn\length{\ell} \opn\pd{pd} \opn\rk{rk}
\opn\projdim{proj\,dim} \opn\injdim{inj\,dim} \opn\rank{rank}
\opn\depth{depth} \opn\grade{grade} \opn\height{height}
\opn\embdim{emb\,dim} \opn\codim{codim}

\opn\Tr{Tr} \opn\bigrank{big\,rank}
\opn\superheight{superheight}\opn\lcm{lcm}
\opn\trdeg{tr\,deg}
\opn\reg{reg} \opn\lreg{lreg} \opn\ini{in} \opn\lpd{lpd}
\opn\size{size} \opn\sdepth{sdepth}
\opn\link{link}\opn\fdepth{fdepth}\opn\lex{lex}
\opn\tr{tr}
\opn\type{type}
\opn\gap{gap}
\opn\diam{diam}
\opn\Mod{Mod}
\opn\div{div} \opn\Div{Div} \opn\cl{cl} \opn\Cl{Cl}
\opn\Spec{Spec} \opn\Supp{Supp} \opn\supp{supp} \opn\Sing{Sing}
\opn\Ass{Ass} \opn\Min{Min}\opn\Mon{Mon}
\opn\Ann{Ann} \opn\Rad{Rad} \opn\Soc{Soc}
\opn\Im{Im} \opn\Ker{Ker} \opn\Coker{Coker} \opn\Am{Am}
\opn\Hom{Hom} \opn\Tor{Tor} \opn\Ext{Ext} \opn\End{End}
\opn\Aut{Aut} \opn\id{id}

\opn\nat{nat}
\opn\pff{pf}
\opn\Pf{Pf} \opn\GL{GL} \opn\SL{SL} \opn\mod{mod} \opn\ord{ord}
\opn\Gin{Gin} \opn\Hilb{Hilb}\opn\sort{sort}
\opn\PF{PF}\opn\Ap{Ap}
\opn\dist{dist}
\opn\aff{aff}
\opn\relint{relint} \opn\st{st}
\opn\lk{lk} \opn\cn{cn} \opn\core{core} \opn\vol{vol}  \opn\inp{inp} \opn\nilpot{nilpot}
\opn\link{link} \opn\star{star}\opn\lex{lex}\opn\set{set}
\opn\width{wd}
\opn\Fr{F}
\opn\QF{QF}
\opn\G{G}
\opn\type{type}\opn\res{res}
\opn\conv{conv}
\opn\sr{sr}
\opn\gr{gr}

\def\pot#1#2{#1[\kern-0.28ex[#2]\kern-0.28ex]}

\opn\dirlim{\underrightarrow{\lim}}
\opn\inivlim{\underleftarrow{\lim}}
\def\Implies{\ifmmode\Longrightarrow \else
	\unskip${}\Longrightarrow{}$\ignorespaces\fi}
\def\implies{\ifmmode\Rightarrow \else
	\unskip${}\Rightarrow{}$\ignorespaces\fi}
\def\iff{\ifmmode\Longleftrightarrow \else
	\unskip${}\Longleftrightarrow{}$\ignorespaces\fi}

\let\:=\colon
\newtheorem{Theorem}{Theorem}[section]
\newtheorem{Lemma}[Theorem]{Lemma}
\newtheorem{Corollary}[Theorem]{Corollary}
\newtheorem{Proposition}[Theorem]{Proposition}

\let\epsilon\varepsilon
\let\kappa=\varkappa
\def\qed{\ifhmode\textqed\fi
	\ifmmode\ifinner\hfill\quad\qedsymbol\else\dispqed\fi\fi}
\def\textqed{\unskip\nobreak\penalty50
	\hskip2em\hbox{}\nobreak\hfill\qedsymbol
	\parfillskip=0pt \finalhyphendemerits=0}
\def\dispqed{\rlap{\qquad\qedsymbol}}

\opn\dis{dis}
\def\pnt{{\raise0.5mm\hbox{\large\bf.}}}

\opn\Lex{Lex}
\opn\Max{Max}
\opn\Shad{Shad}
\opn\astab{astab}
\def\p{\mathfrak{p}}

\def\m{\mathfrak{m}}
\def\bideg{\textup{bideg}}
\opn\v{v}

\begin{document}

	\title{Asymptotic behaviour of the $\mbox{v}$-number of homogeneous ideals}
	\author{Antonino Ficarra, Emanuele Sgroi}
	
	\address{Antonino Ficarra, Department of mathematics and computer sciences, physics and earth sciences, University of Messina, Viale Ferdinando Stagno d'Alcontres 31, 98166 Messina, Italy}
	\email{antficarra@unime.it}
	
	\address{Emanuele Sgroi, Department of mathematics and computer sciences, physics and earth sciences, University of Messina, Viale Ferdinando Stagno d'Alcontres 31, 98166 Messina, Italy}
	\email{emasgroi@unime.it}
	
	\thanks{
	}
	
	
	
	\maketitle
	
	\begin{abstract}
		Let $I$ be a graded ideal of a standard graded polynomial ring $S$ with coefficients in a field $K$. The asymptotic behaviour of the $\v$-number of the powers of $I$ is investigated. Natural lower and upper bounds which are linear functions in $k$ are determined for $\v(I^k)$. We call $\v(I^k)$ the $\v$-function of $I$. We prove that $\v(I^k)$ is a linear function in $k$ for $k$ large enough, of the form $\v(I^k)=\alpha(I)k+b$, where $\alpha(I)$ is the initial degree of $I$, and $b\in\ZZ$ is a suitable integer. For this aim, we construct new blowup algebras associated to graded ideals. Finally, for a monomial ideal in two variables, we compute explicitly its $\v$-function.
	\end{abstract}
	
	\section*{Introduction}
	
	In 1921, Emmy Noether revolutionized Commutative Algebra by establishing the primary decomposition Theorem for Noetherian rings \cite{EN1921}. It says that any ideal $I$ of a Noetherian ring $R$ can be decomposed as the irredundant intersection of finitely many primary ideals $I=Q_1\cap\dots\cap Q_t$ and $\Ass(I)=\{\sqrt{Q_1},\dots,\sqrt{Q_t}\}$, the \textit{set of associated primes of $I$}, is uniquely determined. This fundamental result is a landmark in Commutative Algebra, and always inspires new exciting research trends. A basic question in the seventies was the following. What is the asymptotic behaviour of the set $\Ass(I^k)$ for $k\gg0$ large enough? In 1976, it was predicted by Ratliff \cite{R1976}, and later proved by Brodmann in 1979 \cite{B79}, that $\Ass(I^k)$ stabilizes. That is, there exists $k_0>0$ such that $\Ass(I^{k+1})=\Ass(I^k)$ for all $k\ge k_0$. Another remarkable result of Brodmann says that $\depth(R/I^k)$ is constant for $k\gg0$ \cite{B79a}. Suppose furthermore that $I$ is a graded ideal of a standard graded polynomial ring $S=K[x_1,\dots,x_n]$ with coefficients in a field $K$. In 1999, Kodiyalam \cite{Kod99}, and, independently, Cutkosky, Herzog and Trung \cite{CHT99}, showed that the Castelnuovo--Mumford regularity of $S/I^k$ is a linear function in $k$ for $k\gg0$. The legacy of Brodmann theorem opened up the most flourished research topic in Commutative Algebra: the asymptotic behaviour of the homological invariants of (ordinary) powers of graded ideals, see \cite{CHV2020}.
	
	Now, let $S=K[x_1,\dots,x_n]$ be the standard graded polynomial ring with coefficients in a field $K$, $I\subset S$ be a graded ideal and $\m=(x_1,\dots,x_n)$ be the maximal ideal. Note that $S$ is Noetherian. The graded version of the primary decomposition theorem says that for any prime $\p\in\Ass(I)$, there exists a homogeneous element $f\in S$ such that $(I:f)=\p$. It is natural to define the following invariants. Denote by $S_d$ the $d$th graded component of $S$. The \textit{$\v$-number of $I$ at $\p$} is defined as
	$$
	\v_\p(I)\ =\ \min\{d\ :\ \textit{there exists}\ f\in S_d\ \textit{such that}\ (I:f)=\p\}.
	$$
	Whereas, the \textit{$\v$-number of $I$} is defined as
	$$
	\v(I)\ =\ \min\{d\ :\ \textit{there exists}\ f\in S_d\ \textit{such that}\ (I:f)\in\Ass(I)\}.
	$$
	
	The concept of $\v$-number was introduced by Cooper \textit{et all} in \cite{CSTVV20}, and further studied in \cite{ASS23,BM23,Civan2023,F2023,GMOSV2023,GRV21,JS23,JV21,LSV21,S2023,SS20,VPV23}.\medskip
	
	This invariant plays an important role in Algebraic Geometry and in the theory of (\textit{projective}) \textit{Reed-Muller-type codes} \cite{CSTVV20,DRT2001,GLS2005,camps-sarabia-sarmiento-vila,rth-footprint,sarabia7,GRT}. Let $\mathbb{X}$ be a finite set of points of the projective plane $\mathbb{P}^{s-1}$, and let $\delta_{\mathbb{X}}(d)$ be the minimum distance function of the projective Reed-Muller-type code $C_{\mathbb{X}}(d)$. Then $\delta_{\mathbb{X}}(d)=1$ if and only if $\v(I(\mathbb{X}))\le d$ \cite[Corollary 5.6]{CSTVV20}. In such article, for a radical complete intersection ideal $I$, the famous Eisenbud-Green-Harris conjecture \cite{Eisenbud-Green-Harris} is shown to be equivalent to \cite[Conjecture 6.2]{CSTVV20} (see \cite[Proposition 6.8]{CSTVV20}). This latter conjecture is related to the $\v$-number. Indeed, for such an ideal $I$ we have $\v(I)=\reg(S/I)$. For a nice summary, see \cite[Section 12]{VPV23}.
	
	The $\v$-number of edge ideals was studied in \cite{JV21}. A graph $G$ belongs to the class $W_2$ if and only if $G$ is well--covered without isolated vertices, and $G\setminus v$ is well--covered for all vertices $v\in V(G)$. Let $I(G)\subset S=K[x_v:v\in V(G)]$ be the edge ideal of $G$. Then $G$ is in $W_2$ if and only if $\v(I(G))=\dim(S/I(G))$ \cite[Theorem 4.5]{JV21}. The $\v$-number of binomial edge ideals was recently considered in \cite{ASS23} and \cite{JS23}.\smallskip
	
	In this article, we investigate the eventual behaviour of the function $\v(I^k)$ for $k\gg0$, where $I\subset S$ is a graded ideal. Such a function, for large $k$ measures the ``asymptotic homogeneous growth" of the primary decomposition of $I^k$. We prove that $\v(I^k)=\alpha(I)k+b$ is a linear function for all $k\gg0$, where $\alpha(I)$ is the initial degree of $I$ and $b\in\ZZ$ is a suitable integer.\smallskip
	
	The article is structured as follows. In Section \ref{sec1:FS2}, we recall how to compute the $\v$-number of a graded ideal $I\subset S$ (Theorem \ref{Thm:GRVtheorem}) as shown by Grisalde, Reyes and Villarreal \cite{GRV21}.  For a finitely generated graded $S$-module $M=\bigoplus_{d\ge0}M_d\ne0$, we set $\alpha(M)=\min\{d:M_d\ne0\}$ and $\omega(M)=\max\{d:(M/\m M)_d\ne0\}$. In the next theorem, the bar $\overline{\phantom{p}}$ denotes the residue class modulo $I$.\medskip\\
	{\bf Theorem \ref{Thm:GRVtheorem}}
	\textit{Let $I\subset S$ be a graded ideal and let $\p\in\Ass(I)$. The following hold.
		\begin{enumerate}
			\item[\textup{(a)}] If $\mathcal{G}=\{\overline{g_1},\dots,\overline{g_r}\}$ is a homogeneous minimal generating set of $(I:\p)/I$, then
			$$
			\v_\p(I)=\min\{\deg(g_i)\ :\ 1\le i\le r\ \textit{and}\ (I:g_i)=\p\}.
			$$
			\item[\textup{(b)}] $\v(I)=\min\{\v_\p(I):\p\in\Ass(I)\}$.
			\item[\textup{(c)}] $\v_\p(I)\ge\alpha((I:\p)/I)$, with equality if $\p\in\Max(I)$.
			\item[\textup{(d)}] If $I$ has no embedded primes, then $\v(I)=\min\{\alpha((I:\p)/I):\p\in\Ass(I)\}$.
	\end{enumerate}}\medskip
	
	Firstly, we determine the ``local" numbers $\v_\p(I)$, for all $\p\in\Ass(I)$. After computing a basis of the $S$-module $(I:\p)/I$, we select a generator $\overline{f}\in(I:\p)/I$ of least degree $d$ such that $(I:f)=\p$.  This latter condition is automatically satisfied if $\p\in\Max(I)$. Then $\v_\p(I)=d$ and $\v(I)=\min\{\v_\p(I):\p\in\Ass(I)\}$.
	
	Let $I\subset S$ be a graded ideal, we call $\v(I^k)$ the \textit{$\v$-function} of $I$. In Section \ref{sec2:FS2}, we investigate the asymptotic behaviour of $\v(I^k)$ for $k\gg0$. By Theorem \ref{Thm:GRVtheorem}(b), we have $\v(I^k)=\min\{\v_\p(I^k):\p\in\Ass(I^k)\}$. By Brodmann \cite{B79}, $\Ass(I^k)=\Ass(I^{k+1})$ for all $k\gg0$. We denote this common set by $\Ass^\infty(I)$, and call each prime $\p\in\Ass^\infty(I)$ a \textit{stable prime ideal} of $I$. Moreover, let $\Max^\infty(I)$ be the set of stable prime ideals of $I$ maximal with respect to the inclusion. Thus, $\v(I^k)=\min\{\v_\p(I^k):\p\in\Ass^\infty(I)\}$ for $k\gg0$. To understand the asymptotic behaviour of the $\v$-function, we consider the \textit{$\v_\p$-functions} $\v_\p(I^k)$ for each $\p\in\Ass^\infty(I)$. In the classical case, to prove the asymptotic linearity of the Castelnuovo--Mumford regularity of the powers of $I$, $\reg(I^k)$, one introduces the Rees ring of $I$, $\mathcal{R}(I)=\bigoplus_{k\ge0}I^k$, and shows that this is a bigraded finitely generated module over a suitable polynomial ring \cite{CHT99}.
	
	Let $\p\in\Ass^\infty(I)$. For the $\v_\p$-function $\v_\p(I^k)$, we consider a similar approach as above. It should be noted, however, that the $S$-module $(I^k:\p)/I^k$ has a more subtle module structure than the ordinary power $I^k$. We introduce the module
	$$
	\textup{Soc}_\p(I)\ =\ \bigoplus_{k\ge0}(I^k:\p)/I^k.
	$$
	over the ring $\mathcal{F}_\p(I)\ =\ \bigoplus_{k\ge0}(I^k/\p I^k)$.
	
	A priori, it is not clear that $\textup{Soc}_\p(I)$ is a finitely generated bigraded $\mathcal{F}_\p(I)$-module. This is shown in Theorem \ref{Thm:Soc_p(I)} by carefully analyzing the module structure of $\Soc_\p(I)$. We prove that $\Soc_\p(I)$ is equal to a truncation of the ideal $(0:_{\gr_I(S)}\p)$, and this ideal of $\gr_I(S)$ is finitely generated as a $\mathcal{F}_\p(I)$-module. Here $\gr_I(S)=\bigoplus_{k\ge0}(I^k/I^{k+1})$ denotes the associated graded ring of $I$. The proof relies essentially on a property showed by Ratliff \cite[Corollary 4.2]{R1976}, namely that $(I^{k+1}:I)=I^k$ for all $k\gg0$, and on the fact that $\gr_I(S)$ is Noetherian ring \cite[Proposition (10.D)]{Mat89}.\smallskip
	
	The first main result in Section \ref{sec2:FS2} is\smallskip\\
	{\bf Theorem \ref{Thm:alpha(Ik:p)/Ik}.}
		\textit{Let $I\subset S=K[x_1,\dots,x_n]$ be a graded ideal, and let $\p\in\Ass^\infty(I)$. Then, the following holds.
		\begin{enumerate}
			\item[\textup{(a)}] For all $k\ge1$, we have
			$$
			\alpha((I^k:\p)/I^k)\ \le\ \v_\p(I^k)\ \le\ \omega((I^k:\p)/I^k).
			$$
			\item[\textup{(b)}] The functions $\alpha((I^k:\p)/I^k)$, $\omega((I^k:\p)/I^k)$ are linear in $k$ for $k\gg0$.
			\item[\textup{(c)}] There exist eventually linear functions $f(k)$ and $g(k)$ such that
			$$
			f(k)\le\v(I^k)\le g(k), \ \ \textit{for all}\ \ k\gg0.
			$$
		\end{enumerate}}
	
	Theorem \ref{Thm:alpha(Ik:p)/Ik}(b) follows by a careful analysis of the bigraded structure of $\textup{Soc}_\p(I)$ and in the end boils down to a linear programming argument (Proposition \ref{Prop:LinearProgramm}).
	
	As another remarkable consequence of Theorem \ref{Thm:Soc_p(I)}, if $I\subset S$ is a graded ideal and $\p\in\Ass^\infty(I)$, we show in Proposition \ref{Prop:ItakenOut} that $(I^k:\p)/I^k=(I(I^{k-1}:\p))/I^k$ for all $k\gg0$. In the more general case that $I$ is an ideal of a Noetherian commutative domain $R$ and $\p\in\Ass^\infty(I)$, with an analogous proof, one can also show that  for all $k\gg0$ we have $(I^k:\p)=I(I^{k-1}:\p)$ (Corollary \ref{Cor:ItakenOut-Ratliff}). This result complements the property $(I^k:I)=I^{k-1}$, $k\gg0$, firstly proved by Ratliff \cite[Corollary 4.2]{R1976}. To the best of our knowledge, it was previously unknown.
	
	Our main result in the paper shows that the functions $\v_\p(I^k)$, $\p\in\Ass^\infty(I)$, and $\v(I^k)$ are indeed linear functions in $k$ for all large $k$.\medskip\\
	{\bf Theorem \ref{Thm:v-function-linear}.} \textit{Let $I\subset S=K[x_1,\dots,x_n]$ be a graded ideal. Then, $\v_\p(I^k)$, for all for $\p\in\Ass^\infty(I)$, and $\v(I^k)$ are linear functions in $k$, for $k\gg0$.}\medskip
	
	Let $I\subset S$ be a graded ideal, and let $\p\in\Ass^\infty(I)$. To prove this theorem, we construct another $\F_\p(I)$-module, which we denote by $\Soc_\p^*(I)$, and is defined as a suitable quotient of $\Soc_\p(I)$. In the course of the proof of Theorem \ref{Thm:v-function-linear}, we show that $\v_\p(I^k)=\alpha(\Soc_\p^*(I)_{(*,k)})$. Using again the linear programming argument given in Proposition \ref{Prop:LinearProgramm}, we then conclude that this is a linear function in $k$ for $k\gg0$.
	
	In the last part of the section, we prove some other structural properties of the functions $\v_\p(I^k)$. In Proposition \ref{Prop:vpk+1<k}, we show that $$\v_{\p}(I^{k-1})+\alpha(I)\le\v_\p(I^k)\le\v_{\p}(I^{k-1})+\omega(I),$$ for all $k\gg0$. As nice consequences of this fact, we obtain that the functions $\v_\p(I^k)$, ($\p\in\Ass^\infty(I)$), and $\v(I^k)$ are strictly increasing, for $k\gg0$ (Corollary \ref{Cor:NiceConsequences}).
	
	Due to Theorems \ref{Thm:v-function-linear}, we have that $\v(I^k)=ak+b$ for all $k\gg0$, where $a$ and $b$ are suitable integers. Hence, the limit $\lim_{k\rightarrow\infty}\v(I^k)/k$ exists and it is a finite integer equal to $a$. A natural question arises, what is the number $a$? In Section \ref{sec4:FS2} we address this question, and we obtain the following quite surprising answer:\medskip\\
	{\bf Theorem \ref{Thm:limExistv(I^k)}} \textit{Let $I\subset S$ be a graded ideal. Then
		$$
		\lim_{k\rightarrow\infty}\frac{\v(I^k)}{k}\ =\ \alpha(I).
		$$}\smallskip
	
	The proof of this theorem is based upon Lemma \ref{Lemma:Elementary-a_k}(d). In order to apply it, we find lower and upper bounds for $\v(I^k)$, $k\gg0$, of the form $\alpha(I)k+c$, where $c$ is a suitable constant. The lower bound is easily obtained in Lemma \ref{Lemma:v-funct-lower-bounds}. To obtain the upper bound, we follow an argument given in \cite[Theorem 4.7]{BM23} and generalize a result due to Saha and Sengupta \cite[Proposition 3.11]{SS20}, who proved it only for monomial ideals (Proposition \ref{Prop:General-Saha-Sengupta}). In view of Theorem \ref{Thm:v-function-linear} and \cite{CHT99,Kod99}, we could expect that $\v(I^k)\le\reg(I^k)$ for all $k\gg0$.
	
	Note that the local version of Theorem \ref{Thm:limExistv(I^k)}, namely $\v_\p(I^k)=\alpha(I)k+b$ for all $k\gg0$ and some integer $b$, is false in general, as we show with an example.
	
	Section \ref{sec5:FS2} concerns monomial ideals $I\subset S=K[x,y]$ in two variables. Denote by $G(I)=\{u_1,\dots,u_m\}$ the unique minimal monomial generating set of $I$, with $u_i=x^{a_i}y^{b_i}$ for all $i$. Then $I$ determines the sequences ${\bf a}:a_1>a_2>\dots>a_m$ and ${\bf b}:b_1<b_2<\dots<b_m$, and we set $I=I_{\bf a,b}$. Conversely, given any two such sequences ${\bf a}$ and ${\bf b}$, $\{x^{a_1}y^{b_1},\dots,x^{a_m}y^{b_m}\}$ is the minimal generating set of a monomial ideal of $S$. In terms of ${\bf a}$ and ${\bf b}$ we determine $\Ass^\infty(I_{\bf a,b})$ (Corollary \ref{Cor:Ass(I)K[x,y]}) and the $\v$-number $\v(I_{\bf a,b})$ (Theorem \ref{Thm:vmIab,xy}). Let $\v(I_{\bf a,b}^k)=ak+b$ for $k\gg0$, where $a$ and $b$ are suitable integers. By \cite[Theorem 2.3(c)]{F2023} we have $b\ge-1$. On the other hand, for any integers $a\ge1$ and $b\ge-1$, we are able to construct a monomial ideal of $S$ such that $\v(I^k)=ak+b$ for all $k\ge1$ (Theorem \ref{Thm:v(I^k)=ak+b}).
	
	\section{How to compute the $\v$-number of a graded ideal?}\label{sec1:FS2}
	
	Let $I$ be an ideal of a Noetherian domain $R$. We denote the set of associated primes of $I$ by $\Ass(I)$, and by $\Max(I)$ the set of associated primes of $I$ that are maximal with respect to the inclusion. It is clear that $I$ has no embedded primes if and only if $\Ass(I)=\Max(I)$.

	Let $S=K[x_1,\dots,x_n]=\bigoplus_dS_d$ be the standard graded polynomial ring with $n$ variables and coefficients in a field $K$, and let $\m=(x_1,\dots,x_n)$ be the graded maximal ideal. The concept of $\v$-number was introduced by Cooper \textit{et all} in \cite{CSTVV20}. Let $I\subset S$ be a graded ideal and let $\p\in\Ass(I)$. Then, the \textit{$\v$-number of $I$ at $\p$} is defined as
	$$
	\v_\p(I)\ =\ \min\{d\ :\ \textit{there exists}\ f\in S_d\ \textit{such that}\ (I:f)=\p\}.
	$$
	Whereas, the \textit{$\v$-number of $I$} is defined as
	$$
	\v(I)\ =\ \min\{d\ :\ \textit{there exists}\ f\in S_d\ \textit{such that}\ (I:f)\in\Ass(I)\}.
	$$
	
	Note that if $I=\m=(x_1,\dots,x_n)$, then $\v_\m(I)=\v(I)=0$.
	
	The following result due to Grisalde, Reyes and Villarreal \cite[Theorem 3.2]{GRV21} shows how to compute the $\v$-number of a graded ideal. For a finitely generated graded $S$-module $M=\bigoplus_{d\ge0}M_d\ne0$, we call $\alpha(M)=\min\{d:M_d\ne0\}$ the \textit{initial degree} of $M$. In the next theorem, the bar $\overline{\phantom{p}}$ denotes the residue class modulo $I$.
	\begin{Theorem}\label{Thm:GRVtheorem}
		Let $I\subset S$ be a graded ideal and let $\p\in\Ass(I)$. The following hold.
		\begin{enumerate}
			\item[\textup{(a)}] If $\mathcal{G}=\{\overline{g_1},\dots,\overline{g_r}\}$ is a homogeneous minimal generating set of $(I:\p)/I$, then
			$$
			\v_\p(I)=\min\{\deg(g_i)\ :\ 1\le i\le r\ \textit{and}\ (I:g_i)=\p\}.
			$$
			\item[\textup{(b)}] $\v(I)=\min\{\v_\p(I):\p\in\Ass(I)\}$.
			\item[\textup{(c)}] $\v_\p(I)\ge\alpha((I:\p)/I)$, with equality if $\p\in\Max(I)$.
			\item[\textup{(d)}] If $I$ has no embedded primes, then $\v(I)=\min\{\alpha((I:\p)/I):\p\in\Ass(I)\}$.
		\end{enumerate}
	\end{Theorem}
	
	\section{Asymptotic growth of the modules $(I^k:\p)/I^k$}\label{sec2:FS2}
	
	Let $R$ be a commutative Noetherian domain and $I\subset R$ an ideal. It is known by Brodmann \cite{B79} that $\Ass(I^k)$ stabilizes for large $k$. That is, $\Ass(I^{k+1})=\Ass(I^k)$ for all $k\gg0$. A prime ideal $\p\subset R$ such that $\p\in\Ass(I^k)$ for all $k\gg0$, is called a \textit{stable prime of $I$}.
	
	The set of the stable primes of $I$ is denoted by $\Ass^{\infty}(I)$. Likewise, $\Max^\infty(I)$ denotes the set of stable primes of $I$, maximal with respect to the inclusion. The least integer $k_0$ such that $\Ass(I^{k})=\Ass(I^{k_0})$ for all $k\ge k_0$ is denoted by $\astab(I)$.
	
	Now, let $S=K[x_1,\dots,x_n]$ be the standard graded polynomial ring, with $K$ a field, and unique graded maximal ideal $\m=(x_1,\dots,x_n)$. Let $I\subset S$ be a graded ideal and let $\p\in\Ass^\infty(I)$. In light of Theorem \ref{Thm:GRVtheorem} and Brodmann result, to understand the asymptotic behaviour of the function $\v_\p(I^k)$, one has to understand the asymptotic growth of the modules $(I^k:\p)/I^k$ for $k\gg0$.
	
	Let $M\ne0$ be a finitely generated graded $S$-module. Let $\omega(M)$ be the highest degree of a homogeneous element of the $K$-vector space $M/\m M$. Equivalently, the highest degree $j$ such that the graded Betti number $\beta_{0,j}(M)$ is non-zero. Thus $$\omega(M)=\max\{d:\beta_{0,d}(M)\ne0\}=\max\{d:\Tor^S_0(S/\m,M)_d\ne0\}.$$ Similarly, one has that $\alpha(M)=\min\{d:\Tor^S_0(S/\m,M)_d\ne0\}$.\medskip
	
	The following theorem provides natural asymptotic upper and lower bounds for the $\v$-function $\v(I^k)$ which are linear functions in $k$, for $k\gg0$.
	\begin{Theorem}\label{Thm:alpha(Ik:p)/Ik}
		Let $I\subset S=K[x_1,\dots,x_n]$ be a graded ideal, and let $\p\in\Ass^\infty(I)$. Then, the following holds.
		\begin{enumerate}
			\item[\textup{(a)}] For all $k\ge1$, we have
			$$
			\alpha((I^k:\p)/I^k)\ \le\ \v_\p(I^k)\ \le\ \omega((I^k:\p)/I^k).
			$$
			\item[\textup{(b)}] The functions $\alpha((I^k:\p)/I^k)$, $\omega((I^k:\p)/I^k)$ are linear in $k$ for $k\gg0$.
			\item[\textup{(c)}] There exist eventually linear functions $f(k)$ and $g(k)$ such that
			$$
			f(k)\le\v(I^k)\le g(k), \ \ \textit{for all}\ \ k\gg0.
			$$
		\end{enumerate} 
	\end{Theorem}

	Statement (a) follows immediately from Theorem \ref{Thm:GRVtheorem}(a). Assume for a moment that statement (b) holds, then (c) can be proved as follows. By Brodmann, we have $\v(I^k)=\min\{\v_\p(I^k):\p\in\Ass^\infty(I)\}$ for all $k\gg0$. Thus, by (a), for all $k\gg0$
	$$
	\min_{\p\in\Ass^{\infty}(I)}\alpha((I^k:\p)/I^k)\le\v(I^k)\le\min_{\p\in\Ass^{\infty}(I)}\omega((I^k:\p)/I^k).
	$$
	Setting $f(k)=\min_{\p\in\Ass^{\infty}(I)}\alpha((I^k:\p)/I^k)$ and $g(k)=\min_{\p\in\Ass^{\infty}(I)}\omega((I^k:\p)/I^k)$, by statement (b) it follows that $f(k)$ and $g(k)$ are the required eventually linear functions in $k$. Statement (c) follows.
	
	To prove statement (b), we construct a suitable module that encodes the growth of the modules $(I^k:\p)/I^k$. Indeed, we define it in the following more general context. Let $I$ be an ideal of a commutative Noetherian domain $R$ and let $\p\in\Ass^\infty(I)$. Then we set
	$$
	\Soc_\p(I)\ =\ \bigoplus_{k\ge0}(I^k:\p)/I^k,
	$$
	and $\Soc_\p(I)_k=(I^k:\p)/I^k$ for all $k\ge0$.
	
	The symbol ``$\Soc$" is used, because when $R=S$ or $R$ is local and $\p=\m$ is the (graded) maximal ideal, then $(I^k:\m)/I^k$ is the \textit{socle module} of $S/I^k$, see \cite{CHL}.
	
	The first step consists in showing that $\Soc_\p(I)$ is a finitely generated graded module over a suitable ring. For this aim, we introduce the following ring,
	$$
	\mathcal{F}_\p(I)\ =\ \bigoplus_{k\ge0}(I^k/\p I^k),
	$$
	and we set $\mathcal{F}_\p(I)_k=I^k/\p I^k$. We define addition in the obvious way and multiplication as follows. If $a\in I^k/\p I^k$ and $b\in I^\ell/\p I^\ell$, then $ab\in I^{k+\ell}/\p I^{k+\ell}$. It is routine to check that this multiplication is well--defined.
	
	As before, we note that if $R=S$ or $R$ is local and $\p=\m$ is the maximal ideal, then $\mathcal{F}_\m(I)=\bigoplus_{k\ge0}(I^k/\m I^k)$ is the well--known \textit{fiber cone} of $I$.
	
	With the notation introduced, we have
	\begin{Theorem}\label{Thm:Soc_p(I)}
		Let $I$ be an ideal of a Noetherian commutative domain $R$ and let $\p\in\Ass^{\infty}(I)$. Then, $\Soc_\p(I)$ is a finitely generated graded $\mathcal{F}_\p(I)$-module.
	\end{Theorem}
	\begin{proof}
		Firstly, we show that $\Soc_\p(I)$ has the structure of a graded $\mathcal{F}_\p(I)$-module. For this purpose, let $f\in I^\ell/\p I^\ell$. It is clear that multiplication by $f$ induces a map $(I^k:\p)/I^k\rightarrow(I^{k+\ell}:\p)/I^{k+\ell}$ for any $k\ge0$. Hence $\mathcal{F}_\p(I)_\ell\Soc_\p(I)_k\subseteq\Soc_\p(I)_{k+\ell}$.
		
		To prove that $\Soc_\p(I)$ is a finitely generated $\mathcal{F}_\p(I)$-module, we consider  
		$$
		J=(0:_{\gr_I(R)}\p)=\{f\in\gr_I(R):f\p=0\},
		$$
		\emph{i.e.}, the annihilator of $\p$ in the associated graded ring of $I$, $\gr_I(R)=\bigoplus_{k\ge0}(I^k/I^{k+1})$. Recall that $\gr_I(R)$ is a Noetherian ring \cite[Proposition (10.D)]{Mat89}. Thus, as an ideal of $\gr_I(R)$, $J$ is a finitely generated graded $\gr_I(R)$-module. Since $\p$ annihilates $J$, then $J$ has also the structure of a finitely generated graded $\gr_I(R)/\p\gr_I(R)$-module. But
		\begin{align*}
			\gr_I(R)/\p\gr_I(R)\ &=\ \frac{\bigoplus_{k\ge0}(I^k/I^{k+1})}{\p\bigoplus_{k\ge0}(I^k/I^{k+1})}=\frac{\bigoplus_{k\ge0}(I^k/I^{k+1})}{\bigoplus_{k\ge0}(\p I^k/I^{k+1})}\\&=\ \bigoplus_{k\ge0}\frac{I^k/I^{k+1}}{\p I^k/I^{k+1}}=\bigoplus_{k\ge0}(I^k/\p I^k)\\&=\ \mathcal{F}_\p(I).
		\end{align*}
	    Consequently, $J$ is a finitely generated graded $\mathcal{F}_\p(I)$-module.
	    
	    Let us show that $\Soc_\p(I)_{k+1}=J_k$ for $k\gg0$. For this purpose, we compute the $k$th graded component of $J$. We have
	    \begin{align*}
	    	J_k\ &=\ \{f\in\gr_I(R)_k:f\p=0\}=\{f\in I^k/I^{k+1}\ :f\p=0\}\\
	    	&=\ \{f\in I^k:f\p\in I^{k+1}\}/I^{k+1}=(\{f\in R:f\p\in I^{k+1}\}\cap I^{k})/I^{k+1}\\
	    	&=\ ((I^{k+1}:\p)\cap I^k)/I^{k+1}.
	    \end{align*}
    
    By Ratliff \cite[Corollary 4.2]{R1976}, there exists $r$ such that $(I^{k+1}:I)=I^k$ for all $k\ge r$. Whereas, by Brodmann \cite{B79}, there exists $b$ such that $\Ass(I^{k})=\Ass^{\infty}(I)$ for all $k\ge b$. Let $k^*=\max\{r,b\}$. Next, we show that $\Soc_\p(I)_{k+1}=J_k$ for $k\ge k^*$.
    
    Let $k\ge k^*$. We claim that $\p$ contains $I$. Indeed, $\p\in\Ass(I^k)$, hence $I^{k}\subseteq\p$. Let $a\in I$, then $a^k\in I^k\subseteq\p$. Since $\p$ is prime, actually $a\in\p$ and so $I\subseteq\p$. Therefore, $(I^{k+1}:\p)\subseteq(I^{k+1}:I)=I^k$ by the Ratliff property. Hence,
    $$
    J_k=((I^{k+1}:\p)\cap I^k)/I^{k+1}=(I^{k+1}:\p)/I^{k+1}=\Soc_\p(I)_{k+1}.
    $$
    Consequently, we obtain that $\Soc_\p(I)_{\ge k^*+1}=J_{\ge k^*}$, where $M_{\ge\ell}$ denotes $\bigoplus_{k\ge\ell}M_\ell$ if $M=\bigoplus_{k\ge0}M_k$ is graded. Since $J$ is finitely generated as a $\mathcal{F}_\p(I)$-module, it follows that $\Soc_\p(I)$ is a finitely generated $\mathcal{F}_\p(I)$-module as well.
	\end{proof}
	
	Now, we assume furthermore that $R=S=K[x_1,\dots,x_n]$ is the standard graded polynomial ring with $K$ a field, that $I$ is a graded ideal of $S$ and $\p\in\Ass^\infty(I)$ is a stable prime of $I$. Then, $I^k/\p I^k$ is a graded $S$-module, for all $k\ge0$. Therefore, $\mathcal{F}_\p(I)$ is in a natural way a bigraded ring:
	$$
	\mathcal{F}_\p(I)\ =\ \bigoplus_{d,k\ge0}(I^k/\p I^k)_d.
	$$
	In particular, we set $\mathcal{F}_\p(I)_{(d,k)}=(I^k/\p I^k)_d$ and $\bideg(f)=(d,k)$ for $f\in\mathcal{F}_\p(I)_{(d,k)}$.
	
	Note that each module $(I^k:\p)/I^k$ is a graded $S$-module. Thus, we can write
	$$
	\Soc_\p(I)\ =\ \bigoplus_{d,k\ge0}\Soc_\p(I)_{(d,k)}
	$$
	where $\Soc_\p(I)_{(d,k)}=((I^k:\p)/I^k)_d$. Hence, $\Soc_\p(I)$ is a bigraded $\mathcal{F}_\p(I)$-module, because $\mathcal{F}_\p(I)_{(d_1,\ell)}\Soc_\p(I)_{(d_2,k)}\subseteq\Soc_\p(I)_{(d_1+d_2,k+\ell)}$.\medskip
	
	Therefore, we have proved that
	\begin{Corollary}
		Let $I$ be a graded ideal of $S=K[x_1,\dots,x_n]$ with $K$ a field and let $\p\in\Ass^\infty(I)$. Then, $\Soc_\p(I)$ is a finitely generated bigraded $\mathcal{F}_\p(I)$-module.
	\end{Corollary}
	
	Let $u_1,\dots,u_m$ be a minimal system of homogeneous generators of $I\subset S$. It is well--known that the associated graded ring $\gr_I(S)$ has a presentation
	$$
	\varphi:T=K[x_1,\dots,x_n,y_1,\dots,y_m]\rightarrow\gr_I(S)
	$$
	defined by setting
	\begin{align*}
		\varphi(x_i)=x_i+I\in\gr_I(S)_0=S/I,\ \ &\text{for} \ \ 1\le i\le n,\\
	\varphi(y_i)=u_i+I^2\in\gr_I(S)_1=I/I^2,\ \ &\text{for}\ \ 1\le i\le m.
	\end{align*}
	
	Since $I$ is graded, $\gr_I(S)$ is naturally bigraded, with $\gr_I(S)_{(d,k)}=(I^k/I^{k+1})_d$. Moreover, $T$ can be made into a bigraded ring by setting $\bideg(x_i)=(1,0)$ for $1\le i\le n$, and $\bideg(y_i)=(\deg(u_i),1)$ for $1\le i\le m$, where $\deg(u_i)>0$ is the degree of $u_i$ in $S$. With these bigradings, $\varphi$ is a bigraded surjective ring homomorphism.
	
	In the proof of Theorem \ref{Thm:Soc_p(I)} we have seen that $\mathcal{F}_\p(I)=\gr_I(S)/\p\gr_I(S)$. Let $\pi:\gr_I(S)\rightarrow\mathcal{F}_\p(I)$ be the canonical epimorphism. Then, the composition map 
	\begin{equation}\label{eq:psi-PresentationMap}
		\psi=\pi\circ\varphi\ :\ T\rightarrow\mathcal{F}_\p(I)
	\end{equation}
	is a surjective ring homomorphism. It is clear that $\psi$ preserves the bigraded structure. Thus, $\Soc_\p(I)$ has also the structure of a bigraded $T$-module, if we set
	$$
	af=\psi(a)f \ \ \ \text{for all}\ \ a\in T\ \ \text{and all}\ \ f\in\Soc_\p(I).
	$$
	Since $\psi$ is surjective and $\Soc_\p(I)$ is a finitely generated $\mathcal{F}_\p(I)$-module, it follows that $\Soc_\p(I)$ is a finitely generated $T$-module, as well.
	
	The following lemma is required. For a bigraded $T$-module $M=\bigoplus_{d,k}M_{d,k}$, we set $M_{(*,k)}=\bigoplus_d M_{(d,k)}$. Note that $M_{(*,k)}$ becomes a graded $S$-module.
	\begin{Lemma}
		Let $T=K[x_1,\dots,x_n,y_1,\dots,y_m]$ be a bigraded polynomial ring, with $K$ a field, $\bideg(x_i)=(1,0)$ for $1\le i\le n$ and $\bideg(y_i)=(d_i,1)$ for $1\le i\le m$. Let $\m=(x_1,\dots,x_n)$ and $S=K[x_1,\dots,x_n]\subset T$. Let $M$ be a finitely generated bigraded $T$-module. Then,
		$$
		\Tor^S_i(S/\mathfrak{m},M_{(*,k)})\cong\Tor^T_i(T/\m,M)_{(*,k)}
		$$
		for all $i$ and $k$.
	\end{Lemma}
	\begin{proof}
		Let $\FF:0\rightarrow\cdots\rightarrow F_j\rightarrow\cdots\rightarrow F_1\rightarrow F_0\rightarrow M\rightarrow 0$ be a minimal bigraded $T$-resolution of $M$. Then,
		$$
		\FF_k:0\rightarrow\cdots\rightarrow(F_j)_{(*,k)}\rightarrow\cdots\rightarrow(F_1)_{(*,k)}\rightarrow(F_0)_{(*,k)}\rightarrow M_{(*,k)}\rightarrow 0
		$$
		is a graded (possibly non-minimal) free $S$-resolution of $M_{(*,k)}=\bigoplus_d M_{(d,k)}$. Since $\Tor^T_i(T/\m,M)=H_i(\FF/\m\FF)$ we have that $\Tor^T_i(T/\m,M)_{(*,k)}=H_i(\FF_k/\m\FF_k)$ which in turn is isomorphic to $\Tor^S_i(S/\mathfrak{m},M_{(*,k)})$. The desired conclusion follows.
	\end{proof}
	
	Note that $T/\m=K[y_1,\dots,y_m]$ and that $\Tor^T_0(T/\m,\Soc_\p(I))$ is a finitely generated bigraded $T/\m$-module. Therefore, by the above lemma, we have
	\begin{align*}
		\alpha((I^k:\p)/I^k)&=\alpha(\Tor^S_0(S/\m,(I^k:\p)/I^k))=\alpha(\Tor^S_0(S/\m,\Soc_\p(I)_{(*,k)}))\\&=\alpha(\Tor^T_0(T/\m,\Soc_\p(I))_{(*,k)}).
	\end{align*}
	Similarly, $\omega((I^k:\p)/I^k)=\omega(\Tor^T_0(T/\m,\Soc_\p(I))_{(*,k)})$.\medskip
	
	From this discussion, Theorem \ref{Thm:alpha(Ik:p)/Ik}(b) follows from the next more general statement, which is a variation of \cite[Theorem 3.4]{CHT99}.
	
	\begin{Proposition}\label{Prop:LinearProgramm}
		Let $V=K[y_1,\dots,y_s]$ a polynomial ring, with $\bideg(y_i)=(d_i,1)$, $d_i\ge1$, for $1\le i\le s$ and $K$ a field, and let $M$ be a finitely generated bigraded $V$-module. Then, $\alpha_M(k)=\min\{d:M_{(d,k)}\ne0\}$ and $\omega_M(k)=\max\{d:M_{(d,k)}\ne0\}$ are linear functions in $k$ for $k\gg0$.
	\end{Proposition}
	\begin{proof}
		The claim about the linearity of $\omega_M(k)$ follows from \cite[Theorem 3.4]{CHT99}. The proof of the claim of the linearity of $\alpha_M(k)$ is similar, but we include here all the details for the convenience of the reader. 
		
		For any exact sequence $0\rightarrow M\rightarrow N\rightarrow P\rightarrow 0$ of finitely generated bigraded $V$-modules we have $\alpha_N(k)=\min\{\alpha_{M}(k),\alpha_{P}(k)\}$, for all $k$.
		
		Since $M$ is a finitely generated $V$-module and $V$ is Noetherian, by the bigraded version of \cite[Proposition 3.7]{Ei} there exists a sequence of bigraded $V$-submodules $$0=M_0\subset M_1\subset\cdots\subset M_{i-1}\subset M_i=M$$ of $M$ such that $M_j/M_{j-1}\cong V/\p_j$, with $\p_j$ a bigraded prime ideal of $V$, for all $1\le j\le i$. Hence, we may suppose that $M=V/J$ with $J$ a bigraded ideal of $V$. We show that $J$ can be replaced by a monomial ideal. For this aim, let $>$ be a monomial order on $V$, and let $\textup{in}(J)$ be the initial ideal of $J$ with respect to $>$. The natural $K$-basis of $V/J$ consists of all residue classes (modulo $J$) of all monomials not belonging to $\textup{in}(J)$, see \cite[Proposition 2.2.5.(a)]{JT}. The same residue classes modulo $\textup{in}(J)$ form a $K$-basis for $V/\textup{in}(J)$. Thus $\alpha_{M}(k)=\alpha_{V/J}(k)=\alpha_{V/\textup{in}(J)}(k)$, and we can assume that $M=V/J$ with $J$ a monomial ideal of $V$.
		
		Recall that $\bideg(y_{i})=(d_i,1)$ for $1\le i\le s$. For later use, we may suppose that $d_{1}\le d_{2}\le\cdots\le d_{s}$. Furthermore, we can assume that $J$ is minimally generated by the monomials ${\bf y}^{{\bf c}_i}=y_{1}^{c_{i,1}}y_{2}^{c_{i,2}}\cdots y_{s}^{c_{i,s}}$, for $1\le i\le r$.
		
		Let ${\bf a}=(a_{1},a_{2},\dots,a_{s})\in\NN^{s}$, we denote by $\overline{{\bf y^a}}$ the residue class of ${\bf y^a}=y_{1}^{a_{1}}y_{2}^{a_{2}}\cdots y_{s}^{a_{s}}$ in $V/J$. Let $k\ge0$, by $B_{k}$ we denote the minimal basis of $(V/J)_{k}$. Then, we can write $\alpha_{M}(k)=\min\{v({\bf a}):\overline{{\bf y^a}}\in B_k\}$, where
		$v({\bf a})=\sum_{i=1}^sa_{i}\deg(y_{i})=\sum_{i=1}^sa_{i}d_{i}.$
		
		Clearly, $\overline{{\bf y^a}}\in B_k$ if and only if $\sum_{j=1}^sa_j=k$, and for all $i=1,\dots,s$, there exists $j$ such that $a_{j}<c_{i,j}$. Denote by $L$ the set of all maps $\{1,\dots,r\}\rightarrow\{1,\dots,s\}$. We can decompose the set $B_k$ as the union $\bigcup_{f\in L}B_{k,f}$, where
		$$
		B_{k,f}=\big\{\overline{{\bf y^a}}\ :\ \sum_{j=1}^sa_{j}=k\ \text{and}\ a_{f(i)}<c_{i,f(i)},\ i=1,\dots,r \big\}.
		$$
		With this in mind, we can write $\alpha_M(k)=\min_{f\in L}\alpha_{f}(k)$, where $\alpha_f(k)$ is defined as $\alpha_{f}(k)=\min\{v({\bf a}):\overline{{\bf y^a}}\in B_{k,f}\}$. Hence, it is enough to prove that $\alpha_f(k)$ is a linear function with integer coefficients for all $f\in L$ and all $k\gg0$.
		
		Fix $f\in L$. Let $\{j_1<j_2<\dots<j_t\}$ be the image of $f$. For $h=1,\dots,t$, we set $c_{j_h}=\min\{c_{i,j_h}:i=1,\dots,r\}-1$. Then, we have that
		$$
		B_{k,f}=\big\{\overline{{\bf y^a}}\ :\ \sum_{j=1}^sa_{j}=k\ \text{and}\ a_{j_h}\le c_{j_h},\ h=1,\dots,t\big\}.
		$$
		Thus, $\alpha_f(k)$ is given by the maximum of the \textit{functional} $v({\bf a})$ on the set
		$$
		C_{k,f}=\big\{{\bf a}\ :\ \sum_{j=1}^sa_{j}=k\ \text{and}\ a_{j_h}\le c_{j_h},\ h=1,\dots,t\big\}.
		$$
		
		Let $\ell$ be the smallest integer such that $j_1=1$, $\dots$, $j_{\ell}=\ell$ and $j_{\ell+1}>\ell+1$. Thus, for ${\bf a}=(a_{1},a_{2},\dots,a_{s})\in C_{k,f}$ we have $a_{1}<c_{j_1},\ a_{2}<c_{j_2},\ \dots,\ a_{\ell}<c_{j_{\ell}}$ and no bound on $a_{j_{\ell+1}}$, except that $\sum_{j=1}^sa_j=k$. We distinguish the two possible cases.\medskip\\
		\textsc{Case 1}. Suppose ${\ell}=s$. Then $\sum_{j=1}^s a_{j}$ can be at most $c_{j_1}+c_{j_2}+\dots+c_{j_\ell}$. Thus, for all $k\gg0$, $B_{k,f}=\emptyset$ and so $\alpha_{f}(k)=0$.\medskip\\
		\textsc{Case 2}. Suppose $\ell<s$. We let $k$ such that $k\ge c_{j_1}+c_{j_2}+\dots+c_{j_{\ell}}$. We claim that the functional $v$ has its minimal value for ${\bf a}_*=(c_{j_1},c_{j_2},\dots,c_{j_{\ell}},k-\sum_{p=1}^{\ell}c_{j_p},0,0,\dots,0)$.
		Then, for all large $k\gg0$, we have that
		$$
		\alpha_f(k)=v({\bf a}_*)= \sum_{p=1}^{\ell}c_{j_{p}}d_{j_{p}}+d_{j_{\ell+1}}(k-\sum_{p=1}^{\ell}c_{j_p}),
		$$
		which is a linear function in $k$ with integer coefficients, as desired.
		
		Let ${\bf a}=(a_{1},a_{2},\dots,a_{s})\in C_{k,f}$. Assume that for some $1\le i<j\le s$ we have $a_{i}<c_i$ and $a_{j}>0$. Then, ${\bf a}'=(a_{1},a_{2}\dots,a_{i}+1,\dots,a_{j}-1,\dots,a_{s})$
		also belongs to $C_{k,f}$ and $v({\bf a}')\le v({\bf a})$ because we have $d_{i}\le d_{j}$. Thus, we see that the minimal value of $v$ on $C_{k,f}$ is achieved when we fill up the first ``boxes" of ${\bf a}\in C_{k,f}$ as much as possible. Finally, the functional $v$ reaches its minimal value when ${\bf a}={\bf a}_*$. 
	\end{proof}\smallskip
		
		Theorem \ref{Thm:GRVtheorem}(c) combined with Theorem \ref{Thm:alpha(Ik:p)/Ik}(b) yields
		\begin{Corollary}\label{Cor:linpMax}
			Let $I\subset S=K[x_1,\dots,x_n]$ be a graded ideal and let $\p\in\Max^\infty(I)$. Then, $\v_\p(I^k)$ is a linear function in $k$, for $k\gg0$.
		\end{Corollary}
		
		We conclude this section with the following result that gives an estimate on the growth of the modules $(I^k:\p)/I^k$ for large $k$.
		\begin{Proposition}\label{Prop:ItakenOut}
			Let $I\subset S$ be a graded ideal, and let $\p\in\Ass^\infty(I)$. Then,
			$$
			(I^k:\p)/I^k\ =\ (I(I^{k-1}:\p))/I^k, \ \ \textit{for all}\ k\gg0.
			$$
			In particular, $(I^k:\p)\ =\ I(I^{k-1}:\p)$, for all $k\gg0$.
		\end{Proposition}
	
	In order to prove the above result, we need the next general property.
	\begin{Proposition}\label{Prop:M{(*),k}}
		Let $T=K[x_1,\dots,x_n,y_1,\dots,y_m]$ be a bigraded polynomial ring, with $K$ a field, $\bideg(x_i)=(1,0)$ for $1\le i\le n$ and $\bideg(y_i)=(d_i,1)$, $d_i\ge1$, for $1\le i\le m$. Let $M$ be a finitely generated bigraded $T$-module. Then
		$$
		M_{(*,k)}\ =\ \{y_1,\dots,y_m\}M_{(*,k-1)},\ \ \textit{for all}\ k\gg0.
		$$
		Here $\{y_1,\dots,y_m\}M_{(*,k-1)}$ denotes the set $\{\sum_{i=1}^my_if_i:f_i\in M_{(*,k-1)}\}$.
	\end{Proposition}
	\begin{proof}
		Since $M$ is finitely generated, there exists bihomogeneous elements $g_1,\dots,g_r$ that generate $M$ as a bigraded $T$-module. Let $\bideg(g_i)=(p_i,q_i)$ for $1\le i\le r$.
		
		For vectors ${\bf a}=(a_1,\dots,a_n)\in\ZZ_{\ge0}^n$ and ${\bf b}=(b_1,\dots,b_m)\in\ZZ_{\ge0}^m$, we set
		$$
		{\bf x^a y^b}\ =\ x_{1}^{a_1}\cdots x_n^{a_n}y_1^{b_1}\cdots y_m^{b_m}.
		$$
		Then, these elements form a $K$-basis of $T$, and $\bideg({\bf x^ay^b})=(|{\bf a}|+\sum_{i=1}^mb_id_i,|{\bf b}|)$, where $|{\bf c}|=c_1+\dots+c_t$ is the modulus of ${\bf c}$, if ${\bf c}=(c_1,\dots,c_t)\in\ZZ_{\ge0}^t$ . Hence, any bihomogeneous element $f\in M$ with $\bideg(f)=(d,k)$ can be written as
		\begin{align}\label{eq:f=sum}
			f\ =\ \sum_{i=1}^r(\sum_{{\bf a},{\bf b}}k_{{\bf a},{\bf b},i}{\bf x}^{{\bf a}}{\bf y}^{{\bf b}})g_i,
		\end{align}
		where the sum is taken over all ${\bf a}=(a_{1},\dots,a_{n})\in\ZZ_{\ge0}^n$, ${\bf b}=(b_{1},\dots,b_{m})\in\ZZ_{\ge0}^m$ such that $\bideg({\bf x^ay^b})=(|{\bf a}|+\sum_{j=1}^mb_{j}d_j,|{\bf b}|)=(d-p_i,k-q_i)\in\ZZ_{\ge0}^2$, and $k_{{\bf a},{\bf b},i}\in K$.
		
		Now, let $k>\max\{q_1,\dots,q_r\}$ be an integer. It is clear that $M_{(*,k)}$ contains $\{y_1,\dots,y_m\}M_{(*,k-1)}$. For the reverse inclusion, take $f\in M_{(*,k)}$ bihomogeneous of degree $(d,k)$. Writing $f$ as in (\ref{eq:f=sum}), since $k>q_i$ for all $i$, we see that for all ${\bf a},{\bf b}$ such that the corresponding summand $k_{{\bf a},{\bf b},i}{\bf x^a y^b}$ in (\ref{eq:f=sum}) is non-zero, we have $|{\bf b}|=k-q_i>0$ for all $i$. Hence, ${\bf y}^{{\bf b}}=y_{j_{{\bf a,b},i}}{\bf y}^{{\bf b}-{\bf e}_{j_{{\bf a,b},i}}}$, for some $1\le j_{{\bf a,b},i}\le m$. Here ${\bf e}_{\ell}$ is the vector with all entries zero except for the $\ell$th one which is equal to one.
		
		Therefore, we can write
		$$
		f\ =\ \sum_{i=1}^r(\sum_{{\bf a},{\bf b}}k_{{\bf a},{\bf b},i}{\bf x^ay^b})g_i\ =\ \sum_{i=1}^r\sum_{{\bf a},{\bf b}}y_{j_{{\bf a,b},i}}(k_{{\bf a},{\bf b},i}{\bf x}^{{\bf a}}{\bf y}^{{\bf b}-{\bf e}_{j_{{\bf a,b},i}}}g_i).
		$$
		Since each element $k_{{\bf a},{\bf b},i}{\bf x}^{{\bf a}}{\bf y}^{{\bf b}-{\bf e}_{j_{{\bf a,b},i}}}g_i$ has bidegree $(d-d_{j_{{\bf a,b},i}},k-1)$, from the above equation we see that $f\in\{y_1,\dots,y_m\}M_{(*,k-1)}$. Hence, $M_{(*,k)}\subseteq\{y_1,\dots,y_m\}M_{(*,k-1)}$ and the proof is complete.
	\end{proof}
	
	We are now ready for the proof of the proposition.
	\begin{proof}[Proof of Proposition \ref{Prop:ItakenOut}]
		Taking into account equation (\ref{eq:psi-PresentationMap}) and that $\Soc_\p(I)$ is a finitely generated bigraded $T$-module, by applying Proposition \ref{Prop:M{(*),k}} we obtain that
		$$
		\Soc_\p(I)_{(*,k)}\ =\ \{y_1,\dots,y_m\}\Soc_\p(I)_{(*,k-1)},\ \ \textup{for all}\ k\gg0.
		$$
		Note that $\Soc_\p(I)_{(*,k)}=\bigoplus_{d\ge0}((I^k:\p)/I^k)_d=(I^k:\p)/I^k$, for all $k$. Hence,
		$$
		(I^k:\p)/I^k\ =\ \{y_1,\dots,y_m\}(I^{k-1}:\p)/I^{k-1},\ \ \textup{for all}\ k\gg0.
		$$
		Now, since $\mathcal{F}_\p(I)_{(*,1)}\Soc_\p(I)_{(*,k-1)}\subseteq\Soc_\p(I)_{(*,k)}$ and $\psi(y_i)=u_i+\p I\in\mathcal{F}_\p(I)_{(*,1)}$ for all $1\le i\le m$, we obtain that
		$$
		(I^k:\p)/I^k\ =\ (I(I^{k-1}:\p))/I^{k},\ \ \textup{for all}\ k\gg0.
		$$
		Finally, lifting this equation to $S$ yields $(I^k:\p)=I(I^{k-1}:\p)$, for all $k\gg0$.
	\end{proof}

    With a similar argument, we can show the next analogue of Proposition \ref{Prop:ItakenOut}. This result complements \cite[Corollary 4.2]{R1976} due to Ratliff. To the best of our knowledge, it was previously unknown.
    \begin{Corollary}\label{Cor:ItakenOut-Ratliff}
    	Let $R$ be a commutative Noetherian domain, $I\subset R$ be an ideal, and $\p\in\Ass^\infty(I)$ be a stable prime of $I$. Then,
    	$$
    	(I^k:\p)\ =\ I(I^{k-1}:\p), \ \ \textit{for all}\ k\gg0.
    	$$
    \end{Corollary}
	
	\section{Asymptotic behaviour of the $\v$-number}\label{sec3:FS2}
	
	In this section, we prove the main result in the article: $\v(I^k)$ is indeed linear function in $k$ for all $k\gg0$, and all graded ideals $I\subset S$.
	
	\begin{Theorem}\label{Thm:v-function-linear}
		Let $I\subset S=K[x_1,\dots,x_n]$ be a graded ideal. Then, $\v_\p(I^k)$, for all for $\p\in\Ass^\infty(I)$, and $\v(I^k)$ are linear functions in $k$, for $k\gg0$.
	\end{Theorem}

    In order to prove the theorem, we construct a new module starting from $\Soc_\p(I)$.\medskip
    
    Let $I\subset S$ be a graded ideal, and let $\p\in\Ass^\infty(I)$. For all $k$, let $\overline{g}_{k,1},\dots,\overline{g}_{k,r_k}$ be a minimal homogeneous generating set of $\Soc_\p(I)_{(*,k)}=(I^k:\p)/I^k$. We set
    \begin{align*}
    	A_k\ &=\ \{\overline{g}_{k,j}\ :\ (I^k:g_{k,j})=\p\},\\
    	B_k\ &=\ \{\overline{g}_{k,j}\ :\ (I^k:g_{k,j})\ne\p\}.
    \end{align*} 
    
    Let $\mathcal{N}$ be the submodule of $\Soc_\p(I)$ generated by $\bigcup_kB_k$. Since each element of the set $\bigcup_kB_k$ is bihomogeneous, $\mathcal{N}$ is a bigraded submodule of $\Soc_\p(I)$. We define the following $\F_\p(I)$-module:
    $$
    \Soc_\p^*(I)\ =\ \Soc_\p(I)/\mathcal{N}.
    $$
    By Theorem \ref{Thm:Soc_p(I)} and the fact that $\mathcal{N}$ is a bigraded submodule of $\Soc_\p(I)$, it follows that $\Soc^*_\p(I)$ is a finitely generated bigraded $\F_\p(I)$-module.
    
    Moreover,
    $$
    \Soc_\p^*(I)_{(*,k)}\ =\ \dfrac{(I^k:\p)/I^k}{\mathcal{N}_{(*,k)}},
    $$
    for all $k\ge0$.
    
    Finally, we recall the following basic rules. Let $I,I_1,I_2,\{J_i\}_i$ be ideals of a commutative Noetherian ring $R$ and let $\p$ be a prime ideal of $R$. Then,
    \begin{enumerate}
    	\item[(i)] $(I:\sum_i J_i)=\bigcap_i(I:J_i)$,
    	\item[(ii)] $((I:I_1):I_2)=(I:I_1I_2)$,
    	\item[(iii)] if $\p=\bigcap_i J_i$, then $\p=J_i$ for some $i$.
    \end{enumerate}
    \begin{proof}[Proof of Theorem \ref{Thm:v-function-linear}]
    	Let $\p\in\Ass^\infty(I)$. By Theorem \ref{Thm:GRVtheorem}, for all $k\gg0$ we have $$\v_\p(I^k)=\min_{\overline{g}_{k,j}\in A_k}\deg(g_{k,j}).$$
    	
    	Next, note that $A_k\cup B_k$ generates $\Soc_\p(I)_{(*,k)}$ and $B_k\subseteq\mathcal{N}$. It follows that a minimal homogeneous generating set of $\Soc_\p^*(I)_{(*,k)}$ is given by the non-zero residue classes, modulo $\mathcal{N}_{(*,k)}$, of the elements of $A_k$. We claim that each of these classes is non-zero modulo $\mathcal{N}_{(*,k)}$. Suppose for a contradiction that some $\overline{g}_{k,j}\in A_k$ belongs to $\mathcal{N}_{(*,k)}$. Since $\mathcal{N}$ is generated by $\bigcup_\ell B_\ell$, we can write
    	\begin{equation}\label{eq:ginB_k}
    		\overline{g}_{k,j}\ =\ \sum_{\substack{1\le\ell\le k\\ \overline{g}_{\ell,p}\in B_\ell}}\overline{a}_{k-\ell,p}\overline{g}_{\ell,p},
    	\end{equation}
    	with each $\overline{a}_{k-\ell,p}\in\mathcal{F}_\p(I)_{(*,k-\ell)}=I^{k-\ell}/\p I^{k-\ell}$. Since $\overline{g}_{k,j}\in A_k$ we have $(I^k:g_{k,j})=\p$. By equation (\ref{eq:ginB_k}) we can write $g_{k,j}=\sum a_{k-\ell,p}g_{\ell,p}+z$ where $z$ is a suitable element of $I^k$. Note that $\bigcap(I^k:a_{k-\ell,p}g_{\ell,p})$ is contained in $(I^k:g_{k,j})=\p$. Indeed, take $h\in\bigcap(I^k:a_{k-\ell,p}g_{\ell,p})$. Then $ha_{k-\ell,p}g_{\ell,p}\in I^k$ for all terms in the intersection. Hence $hg_{k,j}=\sum ha_{k-\ell,p}g_{\ell,p}+hz\in I^k$. Since $(I^\ell:g_{\ell,p})\subseteq(I^k:a_{k-\ell,p}g_{\ell,p})$ for all $\ell$ and $p$,
    	\begin{align*}
    		\p\ \subseteq\ \bigcap(I^\ell:g_{\ell,p})\ \subseteq\ \bigcap(I^k:a_{k-\ell,p}g_{\ell,p})\ \subseteq\ \p.
    	\end{align*}
    	Hence $\bigcap(I^\ell:g_{\ell,p})=\p$. By rule (iii), we have $(I^\ell:g_{\ell,p})=\p$ for some $\ell$ and $p$. But this is a contradiction because $\overline{g}_{\ell,p}\in B_\ell$. This proves our claim. As a consequence, we obtain that for all $k\gg0$
    	$$
    	\alpha(\Soc_\p^*(I)_{(*,k)})\ =\ \min_{\overline{g}_{k,j}\in A_k}\deg(g_{k,j})\ =\ \v_\p(I^k).
    	$$
    	Using the map (\ref{eq:psi-PresentationMap}), we see that $\Soc_\p^*(I)$ is a finitely generated bigraded $T$-module. Moreover, $\alpha(\Soc_\p^*(I)_{(*,k)})=\alpha(\Tor^T_0(T/\m,\Soc_\p^*(I))_{(*,k)})$. Applying Proposition \ref{Prop:LinearProgramm}, we conclude that $\v_\p(I^k)=\alpha(\Soc_\p^*(I)_{(*,k)})$ is a linear function in $k$ for $k\gg0$.
    	
    	Finally, $\v(I)=\min_{\p\in\Ass^\infty(I)}\v_\p(I^k)$ is a linear function in $k$ for $k\gg0$, for it is given by the minimum of finitely many eventually linear functions in $k$. 
    \end{proof}
 
    We conclude the section with some results which estimate the growth of the functions $\v_\p(I^k)$, when $k\gg0$.
    \begin{Proposition}\label{Prop:vpk+1<k}
    	Let $I\subset S$ be a graded ideal and let $\p\in\Ass^\infty(I)$. Then,
    	$$
    	\v_\p(I^{k-1})+\alpha(I)\ \le\ \v_\p(I^{k})\ \le\ \v_\p(I^{k-1})+\omega(I),\ \ \textit{for all}\ k\gg0.
    	$$
    	In particular,
    	$$
    	\v(I^{k-1})+\alpha(I)\le\v(I^{k})\le\v(I^{k-1})+\omega(I),\ \ \textit{for all}\ k\gg0.
    	$$
    \end{Proposition}
    \begin{proof}
    		We first prove that $\v_\p(I^{k-1})+\alpha(I)\le\v_\p(I^{k})$, for $k\gg0$. By Proposition \ref{Prop:ItakenOut}, there exists $k_0>0$ such that $(I^k:\p)/I^k\ =\ (I(I^{k-1}:\p))/I^{k}$, for all $k\ge k_0$. Let $k\ge k_0$. Take $\overline{f}\in(I^k:\p)/I^k$ such that $(I^k:f)=\p$ and $\deg(f)=\v_\p(I^k)$. Since $(I^k:\p)=I(I^{k-1}:\p)$, we can write $f=ug$ where $u\in I$ and $g\in(I^{k-1}:\p)$. Then
    		$$
    		\p\subseteq(I^{k-1}:g)\subseteq(uI^{k-1}:ug)=(uI^{k-1}:f)\subseteq(I^k:f)=\p.
    		$$
    		Here we used that $g\in(I^{k-1}:\p)$ and $uI^{k-1}\subseteq I^k$. Hence, $(I^{k-1}:g)=\p$ and so
    		$$
    		\v_\p(I^k)=\deg(f)=\deg(ug)=\deg(g)+\deg(u)\ge\v_\p(I^{k-1})+\alpha(I).
    		$$
    		
    		Now, we prove the inequality $\v_\p(I^{k})\le\v_\p(I^{k-1})+\omega(I)$, for $k\gg0$. By Brodmann and Ratliff, there exists $k^*>0$ such that $\Ass^\infty(I)=\Ass(I^k)$ and $(I^{k}:I)=I^{k-1}$ for all $k\ge k^*$. Fix $k\ge k^*$ and let $\p\in\Ass^\infty(I)$.
    		
    		Let $f\in S$ be a homogeneous element with $(I^{k-1}:f)=\p$ and $\deg(f)=\v_\p(I^{k-1})$. Let $u_1,\dots,u_m$ be a minimal homogeneous generating set of $I$. By rules (ii) and (i),
    		\begin{align*}
    			\p=(I^{k-1}:f)\ &=\ (I^{k}:I):f=(I^{k}:fI)\\ &=\ (I^{k}:\sum_{i=1}^m(fu_i))=\bigcap_{i=1}^m(I^{k}:fu_i).
    		\end{align*}
    		Hence, by rule (iii), we have $\p=(I^{k}:fu_i)$ for some $i$. By the definition of $\v_\p(I^{k})$, this means that $\v_\p(I^{k})\le\deg(fu_i)=\deg(f)+\deg(u_i)\le\v_\p(I^{k-1})+\omega(I)$.
    	\end{proof}
    
        We have the next nice consequences.
        \begin{Corollary}\label{Cor:NiceConsequences}
        	Let $I\subset S$ be a graded ideal, and let $\p\in\Ass^\infty(I)$. Then,
        	\begin{enumerate}
        		\item[\textup{(a)}] The function $\v_\p(I^k)$ is strictly increasing, for all $k\gg0$.
        		\item[\textup{(b)}] The function $\v(I^k)$ is strictly increasing, for all $k\gg0$.
        		\item[\textup{(c)}] $\omega((I^k:\p)/I^k)\ge\alpha((I^{k-1}:\p)/I^{k-1})+\alpha(I)$, for all $k\gg0$.
        		\item[\textup{(d)}] $\alpha((I^{k}:\p)/I^{k})\le\omega((I^{k-1}:\p)/I^{k-1})+\omega(I)$, for all $k\gg0$.
        	\end{enumerate}
        \end{Corollary}
        \begin{proof}
        	(a) By the previous Proposition \ref{Prop:vpk+1<k}, $\v_\p(I^k)\ge\v_\p(I^{k-1})+\alpha(I)$, for all $k\gg0$. Since $\alpha(I)\ge1$, we get $\v_\p(I^{k})>\v_\p(I^{k-1})$ for $k\gg0$, as wanted. Statement (b) follows from (a) and Theorem \ref{Thm:GRVtheorem}(b). Finally, statements (c) and (d) follow from Theorems \ref{Thm:GRVtheorem}(a) and \ref{Thm:alpha(Ik:p)/Ik}(a) combined with Proposition \ref{Prop:vpk+1<k}.
        \end{proof}
    
        The next elementary lemma will be needed several times in the sequel.
        \begin{Lemma}\label{Lemma:Elementary-a_k}
        	Let $\{a_k\}_{k}$ be a numerical sequence. The following statements hold.
        	\begin{enumerate}
        		\item[\textup{(a)}] Suppose that $a_k\ge a_{k-1}+a$, for all $k\gg0$ and some integer $a$. Then, there exists a constant $c$ such that $a_k\ge ak+c$, for all $k\gg0$.
        		\item[\textup{(b)}] Suppose that $a_k\le a_{k-1}+a$, for all $k\gg0$ and some integer $a$. Then, there exists a constant $c$ such that $a_k\le ak+c$, for all $k\gg0$.
        		\item[\textup{(c)}] Suppose that $a_k=a_{k-1}+a$, for all $k\gg0$ and some integer $a$. Then, there exists a constant $c$ such that $a_k=ak+c$, for all $k\gg0$.
        		\item[\textup{(d)}] Suppose that $ak+b\le a_k\le ak+c$, for all $k\gg0$ and some integers $a,b,c$. Then, $\lim\limits_{k\rightarrow\infty}\frac{a_k}{k}=a$.
        	\end{enumerate}
        \end{Lemma}
        \begin{proof}
        	Suppose that $a_k\ge a_{k-1}+a$, for all $k\ge k_0$ and a certain $k_0>0$. Then
        	\begin{align*}
        		a_k\ &\ge\ a_{k-1}+a\ \ge\ a_{k-2}+2a\ \ge\ \cdots\ \ge\ a_{k_0-1}+(k-k_0+1)a\\
        		&=\ ak+(a_{k_0-1}-(k_0-1)a),
        	\end{align*}
        	for all $k\ge k_0$. Setting $c=a_{k_0-1}-(k_0-1)a$, claim (a) follows. Replacing ``$\ge$" with ``$\le$" or ``$=$", we see that (b) and (c) hold as well. For the proof of (d), note that
        	$$
        	\frac{ak+b}{k}\ \le\ \frac{a_k}{k}\ \le\ \frac{ak+c}{k}
        	$$
        	for all $k\gg0$. Thus, taking the limit for $k\rightarrow\infty$, the Squeeze Theorem for numerical sequences implies that $\lim\limits_{k\rightarrow\infty}\frac{a_k}{k}=a$, as wanted.
        \end{proof}

    \section{The limit behaviour of the sequence $\frac{\v(I^k)}{k}$}\label{sec4:FS2}
    
    {Let $I\subset S$ be a graded ideal. Due to Theorem \ref{Thm:v-function-linear}, there exist integers $a$ and $b$ such that  $\v(I^k)=ak+b$ for $k\gg0$. As a consequence, the limit $\lim_{k\rightarrow\infty}\v(I^k)/k$ exists and it is equal to $a$. In this section, we prove the following surprising fact.}
    \begin{Theorem}\label{Thm:limExistv(I^k)}
    	Let $I\subset S$ be a graded ideal. Then
    	$$
    	\lim_{k\rightarrow\infty}\frac{\v(I^k)}{k}\ =\ \alpha(I).
    	$$
    \end{Theorem}\smallskip
    
    In view of statements (a), (b) and (d) of Lemma \ref{Lemma:Elementary-a_k}, to prove the theorem, it is enough to find lower and upper bounds for $\v(I^k)$, which are linear functions in $k$ of the type $\alpha(I)k+c$, for some constant $c$, and for all $k\gg0$.
    
    We begin by providing the linear upper bound.
    \begin{Lemma}\label{Lemma:UpperBoundAlpha(I)}
    	Let $I\subset S$ be a graded ideal. Then, there exists a constant $c$ such that
    	$$
    	\v(I^k)\ \le\ \alpha(I)k+c,\ \ \textit{for all}\ k\gg0.
    	$$
    \end{Lemma}

    The proof of this lemma is based upon the next bound which generalizes a result of Saha and Sengupta \cite[Proposition 3.11]{SS20}.
    \begin{Proposition}\label{Prop:General-Saha-Sengupta}
    	Let $I\subset S$ be a graded ideal and let $f\in S$ be a homogeneous element not belonging to $I$. Then
    	$$
    	\v(I)\ \le\ \v(I:f)+\deg(f).
    	$$
    \end{Proposition}
    \begin{proof}
    	If $(I:f)\in\Ass(I)$, then $\v(I)\le\deg(f)$. Since, also $\v(I:f)\ge0$, the asserted inequality follows. Suppose that $(I:f)\notin\Ass(I)$. Note that $(I:f)$ is a proper ideal, because $f\notin I$. From the short exact sequence
    	$$
    	0\rightarrow S/(I:f)\rightarrow S/I\rightarrow S/(I,f)\rightarrow0
    	$$
    	it follows that $\Ass(I:f)\subseteq\Ass(I)$. Let $g\in S$ be a homogeneous element such that $((I:f):g)\in\Ass(I:f)$ and $\deg(g)=\v(I:f)$. Then $$((I:f):g)=(I:(fg))\in\Ass(I).$$ Thus $\v(I)\le\deg(fg)=\deg(f)+\deg(g)=\deg(f)+\v(I:f)$.
    \end{proof}
    
    Now, we are ready for the proof of Lemma \ref{Lemma:UpperBoundAlpha(I)} which follows closely the argument given in \cite[Theorem 4.7]{BM23}.
    \begin{proof}[Proof of Lemma \ref{Lemma:UpperBoundAlpha(I)}]
    	Let $f\in I$ with $\deg(f)=\alpha(I)$. Then for any $k\ge2$, $f^{k-1}\notin I^k$ by degree reasons. By the previous proposition, we have $\v(I^k)\le\v(I^k:f)+\alpha(I)$. Applying again the proposition we obtain $\v(I^k:f)\le\v(I^k:f^2)+\alpha(I)$ and hence $\v(I^k)\le\v(I^k:f^2)+2\alpha(I)$. Iterating this reasoning, we get
    	\begin{align}\label{eq:v-inequality-SNoetherian}
    		\v(I^k)\ \le\ \alpha(I)k+(\v(I^k:f^{k-1})-\alpha(I)).
    	\end{align}
    	On the other hand, consider the ascending chain
    	$$
    	I\subseteq(I^2:f)\subseteq(I^3:f^2)\subseteq\cdots\subseteq(I^k:f^{k-1})\subseteq\cdots.
    	$$
    	Since $S$ is Noetherian, there exists $k_0>0$ such that $(I^k:f^{k-1})=(I^{k_0}:f^{k_0-1})$, for all $k\ge k_0$. Thus for all $k\ge k_0$, the number $\v(I^k:f^{k-1})-\alpha(I)$ is a constant $c$. Hence, by equation (\ref{eq:v-inequality-SNoetherian}) we see that $\v(I^{k})\le\alpha(I)k+c$, for all $k\gg0$.
    \end{proof}

    Next, combining the second lower bound in Proposition \ref{Prop:vpk+1<k} with Lemma \ref{Lemma:Elementary-a_k}(a) we obtain the next lower linear bound for $\v(I^k)$.
    \begin{Lemma}\label{Lemma:v-funct-lower-bounds}
    	Let $I\subset S$ be a graded ideal. Then, there exists a constant $c$ such that $\v(I^k)\ge\alpha(I)k+c$, for all $k\gg0$.
    \end{Lemma}

    Now, we are ready for the proof of Theorem \ref{Thm:limExistv(I^k)}.
    \begin{proof}[Proof of Theorem \ref{Thm:limExistv(I^k)}]
    	It follows by combining Lemmas \ref{Lemma:UpperBoundAlpha(I)}, \ref{Lemma:v-funct-lower-bounds} and \ref{Lemma:Elementary-a_k}(d).
    \end{proof}
    
    {Let $I\subset S$ be a graded ideal, and let $\p\in\Ass^\infty(I)$. By Theorem \ref{Thm:v-function-linear}, $\v_\p(I^k)$ is also an eventually linear function of the form $a_\p k+b_\p$ for $k\gg0$. By Proposition \ref{Prop:vpk+1<k}, we obtain that $\alpha(I)\le a_\p\le\omega(I)$. However, in contrast to Theorem \ref{Thm:limExistv(I^k)}, $a_\p$ needs not to be equal to $\alpha(I)$.}
    
    For instance, consider the monomial ideal $I=(x^5,x^4y^3,x^2y^4)$ of $S=K[x,y]$. By Proposition \ref{Prop:Ass(I)xy}(a) we have that $\p_x=(x)\in\Ass(I^k)=\Ass^\infty(I)$ for all $k$. However, by Corollary \ref{Cor:pxpylinV}(a) we have that $\v_{\p_x}(I^k)=\deg(x^2y^4)k-1=6k-1$, for all $k\ge1$, and the slope of this linear function is $6>5=\alpha(I)$.
    
    \section{The $\v$-number of monomial ideals in two variables}\label{sec5:FS2}
    In this section, we consider monomial ideals of the polynomial ring in two variables $S = K[x,y]$. Let $I\subset S$ be a monomial ideal. As customary, we denote by $G(I)$ the unique minimal monomial generating set of $I$. Then
    \[
    G(I)\ =\ \{x^{a_1}y^{b_1},x^{a_2}y^{b_2},\dots,x^{a_m}y^{b_m}\},
    \]
    where ${\bf a}:a_1>a_2>\cdots > a_m\ge 0$ and ${\bf b}:0\le b_1<b_2<\cdots<b_m$. Conversely, given any two such sequences ${\bf a}$ and ${\bf b}$, the set $\{x^{a_1}y^{b_1},x^{a_2}y^{b_2},\dots,x^{a_m}y^{b_m}\}$ is the minimal monomial generating set of a monomial ideal of $S$.
    
    Therefore, the monomial ideals of $S=K[x,y]$ are in bijection with all pairs $({\bf a},{\bf b})$ of sequences ${\bf a}:a_1>a_2>\cdots > a_m\ge 0$ and ${\bf b}:0\le b_1<b_2<\cdots<b_m$ as above. Hereafter, we write $I=I_{\bf a,b}$ for $I=(x^{a_1}y^{b_1},x^{a_2}y^{b_2},\dots,x^{a_m}y^{b_m})$.\smallskip
    
    The natural $K$-basis of $S/I_{\bf a,b}$ consists of the residue classes (modulo $I_{\bf a,b}$) of the monomials not belonging to $I_{\bf a,b}$. These basis elements can be represented by the lattice points $(c,d)\in\ZZ_{\ge0}\times\ZZ_{\ge0}$ such that $x^cy^d\notin I_{\bf a,b}$, as in the next picture.
    
    
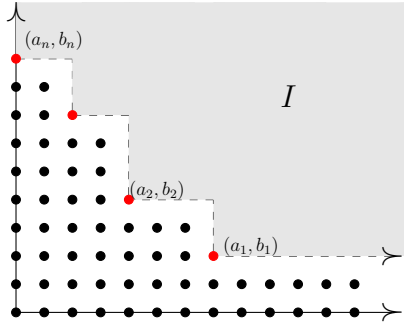
\begin{figure}[H]
    	\input{staircasediagram.tikz}
    \end{figure}\smallskip
    
    {In this section we investigate the $\v$-function of $I_{\bf a,b}$ in terms of ${\bf a}$ and ${\bf b}$.}
	
	The associated prime ideals of a monomial ideal $I$ are monomial prime ideals, that is, ideals generated by a subset of the variables \cite[Corollary 1.3.9]{JT}. Thus, in our case $\Ass(I)\subseteq\{(x),(y),(x,y)\}$. We set $\p_x=(x)$, $\p_y=(y)$ and $\m=(x,y)$.
	
	We can compute $\Ass(I_{\bf a,b})$ in terms of the sequences ${\bf a}$ and ${\bf b}$.
	\begin{Proposition}\label{Prop:Ass(I)xy}
		Let $I=I_{\bf a,b}\subset S$ be a monomial ideal. Then,
		\begin{enumerate}
			\item[\textup{(a)}] $\p_x\in\Ass(I)$ if and only if $a_m>0$.
			\item[\textup{(b)}] $\p_y\in\Ass(I)$ if and only if $b_1>0$.
			\item[\textup{(c)}] $\m\in\Ass(I)$ if and only if $m>1$, i.e., $I$ is not a principal ideal.
		\end{enumerate}
	\end{Proposition}
	\begin{proof}
		If $\p_x\in\Ass(I)$ then $I\subseteq\p_x=(x)$. Hence $x$ divides all minimal monomial generators of $I$. In particular, $x$ divides $x^{a_m}y^{b_m}$ and $a_m>0$.
		
		Conversely, let $a_m>0$. Then $x$ divides all minimal monomial generators of $I$. Hence $I\subseteq\p_x=(x)$. Since $\p_x$ is of height one, it follows that $\p_x\in\Ass(I)$.
		
		This proves (a), statement (b) can be proved similarly.
		
		Finally, for the proof of (c), suppose $\m\in\Ass(I)$. If $I$ is a principal ideal, then for some $c$ and $d$, $I=(x^cy^d)=(x^c)\cap(y^d)=\p_x^c\cap\p_y^d$ is the primary decomposition of $I$, which contradicts our assumption. Hence $I$ is not principal.
		
		Conversely, suppose $I$ is not principal, but $\m\notin\Ass(I)$. Then $\Ass(I)\subseteq\{\p_x,\p_y\}$. Since $\p_x$ and $\p_y$ are height one prime ideals, $I=\p_x^c\cap\p_y^d=(x^c)\cap(y^d)=(x^cy^d)$ for some $c$ and $d$, against our assumption. The assertion follows.
	\end{proof}
	
	The following lemma is required.
	\begin{Lemma}\label{Lemma:(xa1yb1)k_(xamybm)k}
		Let $I=I_{\bf a,b}\subset S$ be a monomial ideal. Then,
		$$
		x^{ka_1}y^{kb_1},x^{ka_m}y^{kb_m}\in G(I^k),\ \ \textit{for all}\ \ k\ge1.
		$$
	\end{Lemma}
	\begin{proof}
		Let $k\geq 1$. We know that
		$$
		I^k\ =\ (\prod_{i=1}^m (x^{a_i}y^{b_i})^{k_i}\ :\ \sum_{i=1}^m k_i = k).
		$$
		Let $x^r y^s$ be an arbitrary generator of $I^k$ different from $x^{ka_1}y^{kb_1}$. Then, we have $r = k_1 a_1 + \dots k_m a_m$, $s = k_1 b_1 + \dots k_m b_m$, $\sum_{i=1}^m k_i = k$ and $k_i>0$ for some $i\ne1$. Thus,
		$$
		ka_1 = k_1a_1+k_2 a_1+\dots k_m a_1 > k_1 a_1 + k_2 a_2 + \dots k_m a_m = r
		$$
		and
		$$
		kb_1 = k_1b_1+k_2 b_1+\dots k_m b_1 < k_1 b_1 + k_2 b_2 + \dots k_m b_m = s.
		$$
		Therefore, $x^ry^s$ does not divide $x^{ka_1}y^{kb_1}$. This shows that $x^{ka_1}y^{kb_1}$ is a minimal generator of $I^k$. By a similar argument we obtain that $x^{ka_m}y^{kb_m}\in G(I^k)$.
	\end{proof}
	
	\begin{Corollary}\label{Cor:Ass(I)K[x,y]}
		Let $I=I_{\bf a,b}\subset S$ be a monomial ideal. Then $\Ass(I^k)=\Ass^\infty(I)$, for all $k\ge1$. In particular, $\textup{astab}(I)=1$.
	\end{Corollary}
	\begin{proof}
		Let us prove that $\Ass(I)=\Ass(I^k)$ for all $k\ge2$. By the previous proposition, $\p_x\in\Ass(I)$ if and only if $x$ divides all minimal monomial generators of $I$. Hence, if $\p_x\in\Ass(I)$, then $\p_x\in\Ass(I^k)$ for all $k\ge2$, as well.
		
		Now, suppose that $\p_x\in\Ass(I^k)$ for some $k\ge2$, but $\p_x\notin\Ass(I)$. Then, by Proposition \ref{Prop:Ass(I)xy}(a) we have $a_m=0$. Hence $y^{b_m}\in G(I)$. By the previous corollary, $y^{kb_m}\in G(I^k)$, as well. But this is impossible, because $y^{kb_m}\notin\p_x$, but by assumption $\p_x$ contains $I^k$. Hence $a_m>0$ and $\p_x\in\Ass(I)$, as wanted.
		
		The same reasoning can be applied to show that $\p_y\in\Ass(I)$ if and only if $\p_y\in\Ass(I^k)$, for any $k\ge2$.
		
		Finally, by the previous proposition, $\m\in\Ass(I)$ if and only $I$ is not principal. Lemma \ref{Lemma:(xa1yb1)k_(xamybm)k} implies that $I$ is not principal if and only if $I^k$ is not principal for any $k\ge2$. Thus $\m\in\Ass(I)$ if and only if $\m\in\Ass(I^k)$ for any $k\ge2$.
	\end{proof}
	
	Next, we compute the functions $\v_{\p_x}(I_{\bf a,b}^k)$, $\v_{\p_y}(I_{\bf a,b}^k)$.
	\begin{Corollary}\label{Cor:pxpylinV}
		Let $I=I_{\bf a,b}\subset S$ be a monomial ideal. The following holds.
		\begin{enumerate}
			\item[\textup{(a)}] If $a_m>0$, then $\v_{\p_x}(I^k)=k(a_m+b_m)-1$, for all $k\ge1$.
			\item[\textup{(b)}] If $b_1>0$, then $\v_{\p_y}(I^k)=k(a_1+b_1)-1$, for all $k\ge1$.
		\end{enumerate}
	\end{Corollary}
    \begin{proof}
    	The only generator $\overline{u}\in(I:\p_x)/I$ such that $(I:u) = \p_x$ is $x^{a_m-1}y^{b_m}$, for it has the largest $y$-degree. For each $k\geq 1$, from Lemma \ref{Lemma:(xa1yb1)k_(xamybm)k}, $x^{ka_m}y^{kb_m}\in G(I^k)$ and such generator has the highest $y$-degree. Thus, $\overline{u}= \overline{x^{ka_m-1}y^{kb_m}}$ is the only generator of $(I^k:\p_x)/I^k$ such that $(I^k:u)=\p_x$. Similarly, one can prove (b).
    \end{proof}
	
    In the next proposition, we show how to compute $\v_\m(I)$ for a non principal monomial ideal $I=I_{\bf a,b}\subset S$. For our convenience, if $c\ge1$, in the proof of the next result we regard $x^{-c}$ and $y^{-c}$ as $1$.
    
    \begin{Proposition}\label{2VarPropColon(x,y)}
    	Let $I=I_{\bf a,b}\subset S$ be a non principal monomial ideal. Then,
    	\begin{equation}\label{eq:I:m/Ixy}
    		(I :\m)/I = (x^{a_j-1}y^{b_{j+1}-1}\ :\ 1\le j\le m-1)/I.
    	\end{equation}
    	In particular,
    	$$
    	\v_\m(I)=\min\{a_j+b_{j+1}-2\ :\ 1\le j\le m-1\}.
    	$$
    \end{Proposition}
    \begin{proof}
    	Firstly, we compute $I:\m$. We have
    \begin{equation}\label{eq:I:mxy1}
    	\begin{aligned}
    	I:\m\ &=\ (I:\p_x)\cap (I:\p_y)\\
    	&=\ (x^{a_1-1}y^{b_1},\dots, x^{a_m-1}y^{b_m})\cap (x^{a_1}y^{b_1-1},\dots, x^{a_m}y^{b_m-1})\\
    	&=\ (\text{lcm}(x^{a_i-1}y^{b_i},x^{a_j}y^{b_j-1})\ :\ 1\le i\le m,\ 1\le j\le m)\\
    	&=\ (\text{lcm}(x^{a_i-1}y^{b_i},x^{a_{j+1}}y^{b_{j+1}-1})\ : \
    	1\le i\le m,\ 0\le j\le m-1)\\
    	&=\ (x^{\max\{a_i-1,a_{j+1}\}}y^{\max\{b_i,b_{j+1}-1\}}\ : \
    	1\le i\le m,\ 0\le j\le m-1).
    \end{aligned}
    \end{equation}

    Fix $j\in\{0,\dots,m-1\}$ and let $i\in\{1,\dots,m\}$.
    
    If $i\le j$, we have $a_i\ge a_j>a_{j+1}$ and $b_i\le b_j<b_{j+1}$. Therefore $x^{a_{j+1}} | x^{a_i-1}$ and $y^{b_i} | y^{b_{j+1}-1}$. Hence,
    \begin{equation}\label{eq:I:mxy2}
    	x^{a_j-1}y^{b_{j+1}-1}\in(I:\m)\ \ \text{divides} \ \ x^{\max\{a_i-1,a_{j+1}\}}y^{\max\{b_i,b_{j+1}-1\}}\ \ \text{for}\ \ i\le j.
    \end{equation}
    
    If $i>j$, we have $a_i-1\le a_{j+1}$ and $b_i\ge b_{j+1}-1$, so $x^{a_i-1} | x^{a_{j+1}}$ and $y^{b_{j+1}-1} | y^{b_i}$. Hence,
    \begin{equation}\label{eq:I:mxy3}
    	x^{a_{j+1}}y^{b_i}\in(I:\m)\ \ \text{divides} \ \ x^{\max\{a_i-1,a_{j+1}\}}y^{\max\{b_i,b_{j+1}-1\}}\ \ \text{for}\ \ i> j.
    \end{equation}
    
    Thus, by equations (\ref{eq:I:mxy1}), (\ref{eq:I:mxy2}) and (\ref{eq:I:mxy3}) we have
    \[
    I:\m\ =\ (x^{a_j-1}y^{b_{j+1}-1}, x^{a_{j+1}}y^{b_i}\ :\ 1\le j\le m-1,\ j+1\le i\le m).
    \]
    Note that for each $i\ge j+1$ we have $x^{a_{j+1}}y^{b_i}\in I$. It is clear that $x^{a_j-1}y^{b_{j+1}-1}\notin I$, for all $j=1,\dots,m-1$. Hence, equation (\ref{eq:I:m/Ixy}) follows.
    
    The claim about $\v_\m(I)$ follows from (\ref{eq:I:m/Ixy}) and Theorem \ref{Thm:GRVtheorem}(c).
    \end{proof}
    
    As a consequence of our discussion, we obtain the next formula that shows us how to compute the $\v$-number of $I_{\bf a,b}$ solely in terms of the sequences ${\bf a}$ and ${\bf b}$.
    \begin{Theorem}\label{Thm:vmIab,xy}
    	Let $I=I_{\bf a,b}\subset S$ be a monomial ideal. Then
    	\[
    	\v(I)=
    	\begin{cases}
    		\min\{a_i+b_{i+1}-2\ :\ 1\le i\le m-1\},\ \text{if}\ b_1=0\ \text{and}\ a_m=0,\\
    		\min\{a_1+b_1-1, a_i+b_{i+1}-2\ :\ 1\le i\le m-1\},\ \text{if}\ b_1\neq0\ \text{and}\ a_m=0,\\
    		\min\{a_m+b_m-1, a_i+b_{i+1}-2\ :\ 1\le i\le m-1\},\ \text{if}\ b_1=0\ \text{and}\ a_m\neq 0,\\
    		\min\{a_1+b_1-1,a_m+b_m-1, a_i+b_{i+1}-2\ :\ 1\le i\le m-1\},\ \text{otherwise.}
    	\end{cases}
    	\]
    \end{Theorem}
	\begin{proof}
		Suppose $I$ is non principal. Then, the statement follows by combining Proposition \ref{Prop:Ass(I)xy}, Corollary \ref{Cor:pxpylinV} and Proposition \ref{2VarPropColon(x,y)}. Now, if $I$ is principal, the above formulas also hold. Indeed, in this case $m=1$ and in the above last three minimums one does not have to consider the terms $a_i+b_{i+1}-2$ because $m-1=0$.
	\end{proof}
    
    Let $I\subset S=K[x_1,\dots,x_n]$ be a monomial ideal. Then $$\v(I^k)=\alpha(I)k+b,$$ for all $k\gg0$. By \cite[Theorem 2.3(c)]{F2023} we have that $b\ge-1$. Note also that $a=\alpha(I)\ge1$. On the other hand, let $f(k)=ak+b$ be any linear function in $k$, with $a\ge1$ and $b\ge-1$. In the next theorem, we prove that there exists a monomial ideal $I\subset K[x,y]$ such that $\v(I^k)$ agrees with $f(k)$ for all $k\ge1$.
    
    \begin{Theorem}\label{Thm:v(I^k)=ak+b}
    	Let $a\ge1$ and $b\ge-1$ be integers. Then, there exists a monomial ideal $I\subset S=K[x,y]$ such that
    	$$
    	\v(I^k)=ak+b,\ \ \textit{for all}\ \ k\ge1.
    	$$
    \end{Theorem}
    \begin{proof}
    	We claim that $I = (x^a, x^{a-1}y^{b+2})=x^{a-1}(x,y^{b+2})$ satisfies our assertion. For this aim, let us show that $I^k = (x^{ka-i}y^{i(b+2)} : 0\le i \le k)$ for all $k\ge1$. Indeed,
    	\begin{eqnarray*}
    	I^k &=& x^{k(a-1)}(x,y^{b+2})^k=x^{k(a-1)}\sum_{i=0}^k(x^{k-i}y^{i(b+2)})\\
    	&=& (x^{ka-i}y^{i(b+2)} : 0\le i \le k).
        \end{eqnarray*}
        Since $ka>ka-1>\dots>ka-k$ and $b+2<2(b+2)<\dots<k(b+2)$, it follows that $G(I^k)=\{x^{ka-i}y^{i(b+2)} : 0\le i \le k\}$.\smallskip
        
        Note that $\Ass^\infty(I)=\{\p_x,\m\}$ if $a>1$, and $\Ass^\infty(I)=\{\m\}$ if $a=1$.
        
        If $a>1$, by Corollary \ref{Cor:pxpylinV}(a) we have $\v_{\p_x}(I^k)=k(a+b+1)-1$.
        
        Whereas, by Proposition \ref{2VarPropColon(x,y)},
        \begin{align*}
        	\v_\m(I^k)\ &=\ \min\{(ka-i)+(i+1)(b+2)-2:0\le i\le k-1\}\\
            &=\ \min\{ka+(i+1)b+i:0\le i\le k-1\}\\
            &=\ ak+b.
        \end{align*}
    	If $a>1$, then $\v(I^k)=\min\{\v_{\p_x}(I^k),\v_{\m}(I^k)\}=\min\{k(a+b+1)-1,ak+b\}=ak+b$. Otherwise, if $a=1$, then $\v(I^k)=\v_\m(I^k)=ak+b$ once again.
    \end{proof}

\end{document}